\documentclass[12pt]{amsart}
\usepackage{amsmath}
\usepackage{amsfonts}
\usepackage{amssymb}
\usepackage{pgf,tikz,pgfplots}
\pgfplotsset{compat=1.13}
\usepackage{mathrsfs}
\usetikzlibrary{arrows}

\newcommand{\abs}[1]{\ensuremath{\left| #1 \right| }}
\newcommand{\newD}{D^+_{\mathrm{circ}}}
\newcommand{\newDD}{{\tilde{D}}^+_{\mathrm{circ}}}

\newcommand{\Rst}{\mathbb{R}}

\newcommand{\bR}{\mathbb{R}}

\newtheorem{lemma}{Lemma}
\newtheorem{theorem}{Theorem}

\author{Jos\'e Luis Romero}
\address{Faculty of Mathematics \\
University of Vienna \\
Oskar-Morgenstern-Platz 1 \\
A-1090 Vienna, Austria \\and
Acoustics Research Institute\\ Austrian Academy of Sciences\\Wohllebengasse 12-14, Vienna, 1040, Austria}
\email{jose.luis.romero@univie.ac.at, jlromero@kfs.oeaw.ac.at}

\thanks{J. L. R. gratefully acknowledges support from the Austrian Science Fund (FWF): Y 1199 and P 29462.}

\title{Sign retrieval in shift-invariant spaces with totally positive generator}

\begin{document}
\begin{abstract}
We show that a real-valued function $f$ in the shift-invariant space generated by a totally positive function of Gaussian type is uniquely determined, up to a sign, by its absolute values $\{|f(\lambda)|: \lambda \in \Lambda \}$ on any set $\Lambda \subseteq \mathbb{R}$ with lower Beurling density $D^{-}(\Lambda)>2$.
\end{abstract}
\maketitle\thispagestyle{empty}

We consider a \emph{totally positive function of Gaussian type}, i.e., a function $g \in L^2(\mathbb{R})$ whose Fourier transform factors as
\begin{equation}\label{eq_tpgt}
\hat g(\xi)= 
\int_{\mathbb{R}} g(x) e^{-2\pi i x \xi} dx =
C_0 e^{- \gamma \xi^2}\prod_{\nu=1}^m (1+2\pi i\delta_\nu \xi)^{-1},
\qquad
\xi \in \Rst,
\end{equation}
with $\delta_1,\ldots,\delta_m\in \mathbb{R}, C_0, \gamma >0, m \in \mathbb{N} \cup \{0\}$, and the \emph{shift-invariant space}
\begin{align*}
V^\infty(g) = \Big\{ f=\sum_{k \in \mathbb{Z}} c_k\, g(\cdot-k): c \in \ell^\infty(\mathbb{Z}) \Big\},
\end{align*}
generated by its integer shifts within $L^\infty(\mathbb{R})$. As a consequence of \eqref{eq_tpgt}, each $f \in V^\infty(g)$ is continuous, the defining series converges unconditionally in the weak$^*$ topology of $L^\infty$, and the coefficients $c_k$ are unique \cite[Theorem 3.5]{jm91}.

The shift-invariant space $V^\infty(g)$ enjoys the following sampling property \cite{GRS18}:
every separated set $\Lambda \subseteq \mathbb{R}$ with \emph{lower Beurling density}
\begin{align*}
D^{-}(\Lambda) := \liminf_{r \longrightarrow \infty} 
\inf_{x\in\mathbb{R}}
\frac{\# \big[ \,\Lambda \cap [x-r,x+r] \,\big]}{2r}
\end{align*}
strictly larger than $1$ provides the norm equivalence
\begin{align}\label{eq_samp}
||{f}||_{L^\infty(\mathbb{R})} \asymp ||{f|\Lambda}||_{\ell^\infty({\Lambda})},
\qquad f \in V^\infty(g).
\end{align}
(A similar property holds for all $L^p$ norms, $1 \leq p \leq \infty$; see \cite[Theorems 5.2 and 3.1]{GRS18}.)

In this article, we show the following uniqueness property for the absolute values of real-valued functions in $V^\infty(g)$.
\begin{theorem}[Sign retrieval]
	\label{th_main}
	Let $g$ be a totally positive function of Gaussian type, as in \eqref{eq_tpgt}, and
	$\Lambda \subseteq \mathbb{R}$ with lower Beurling density
	\begin{align}\label{eq_d_2}
	D^{-}(\Lambda)>2.
	\end{align}
	Assume that $f_1, f_2 \in V^\infty(g)$ are real-valued
	and $\abs{f_1} \equiv \abs{f_2}$ on $\Lambda$. Then either $f_1 \equiv f_2$ or $f_1 \equiv - f_2$.
\end{theorem}
See \cite{MR3977121} for motivation for sign retrieval in shift-invariant spaces. When $g$ is a Gaussian function - corresponding to $m=0$ in \eqref{eq_tpgt} - Theorem \ref{th_main} was recently obtained in \cite[Theorem 1]{gro20}. Here, that result is extended to all totally positive functions of Gaussian type.\footnote{In contrast to \cite[Theorem 1]{gro20}, in Theorem \ref{th_main} it is not assumed that $\Lambda$ is separated, or even relatively separated. See \cite{gro20} for an example of a separated set with $D^{-}(\Lambda)<2$ for which sign retrieval fails in a shift-invariant space with Gaussian generator.}

The intuition behind Theorem \ref{th_main} is as follows. Suppose that $\abs{f_1} \equiv \abs{f_2}$ on $\Lambda$, and split $\Lambda$ into
\begin{align}\label{eq_split}
\begin{aligned}
\Lambda_1 &= \{ \lambda \in \Lambda: f_1(\lambda) = f_2(\lambda)\},
\\
\Lambda_2 &= \{ \lambda \in \Lambda: f_1(\lambda) = -f_2(\lambda)\}.
\end{aligned}
\end{align}
Then $f_1 - f_2$ vanishes on $\Lambda_1$ while $f_1 + f_2$ vanishes on $\Lambda_2$. Under \eqref{eq_d_2}, one may expect one of the two subsets $\Lambda_j$ to have lower Beurling density larger than 1. The sampling inequalities \eqref{eq_samp} would then imply that either $f_1 - f_2$ or $f_1 + f_2$ are identically zero. This argument breaks down, however, because Beurling's lower density is not subadditive. For example, $\Lambda_1 := \Lambda \cap (-\infty,0]$ and $\Lambda_2 := \Lambda \cap (0,\infty)$ have always zero lower Beurling density. The proof of the sign retrieval theorem
for Gaussian generators in \cite{gro20} resorts instead to a special property of the Gaussian function, namely that $V^\infty(g) \cdot V^\infty(g)$ is contained in a dilation of $V^\infty(\tilde{g})$ by a factor of $2$, where $\tilde{g}$ is another Gaussian function. Thus, in the Gaussian case, the sampling theorem can be applied after rescaling to the set $\Lambda$ to conclude that $(f_1-f_2) \cdot (f_1+f_2) \equiv 0$, and, by analyticity, that either $f_1 \equiv f_2$ or $f_1 \equiv - f_2$. A similar argument applies to Paley-Wiener spaces \cite[Theorem 2.5]{MR3656501}. We are unaware of an analogous dilation property for general totally positive generators.

To prove Theorem \ref{th_main} for all totally positive generators of Gaussian type we take a different route. We define the \emph{upper average circular density} of a set $\Lambda \subseteq \mathbb{R}$ as
\begin{align}\label{eq_newd}
\newD(\Lambda) = \limsup_{r \longrightarrow \infty} \frac{4}{\pi r^2}
\int_0^r \sum_{\lambda \in \Lambda \cap [-t,t]} 
\sqrt{t^2 - \lambda^2} \,
\frac{dt}{t},
\end{align}
with the convention that $\newD(\Lambda)=\infty$, if $\Lambda$ is uncountably infinite. The density is named circular because, in \eqref{eq_newd}, each point $\lambda \in [-t,t]$ is weighted with the measure of the largest vertical segment $\{\lambda\} \times (-a,a)$ contained in the two-dimensional open disk $B_t(0) \subseteq \mathbb{R}^2$.

The upper average circular density can be alternatively described as follows:
as shown in Lemma \ref{lemma_circ} below, for any lattice $\alpha \mathbb{Z}$, $\alpha>0$,
\begin{align}\label{eq_dens_2}
\newD(\Lambda)= \limsup_{r \longrightarrow \infty}
\frac{2\alpha}{\pi r^2} \int_0^r \# \big[ (\Lambda \setminus \{0\} \times \alpha \mathbb{Z}) \cap B_t(0)\big] \,\frac{dt}{t}.
\end{align}
From here, it follows easily that $\newD$ dominates Beurling's lower density:
\begin{align}\label{eq_dens_3}
\newD(\Lambda) \geq D^{-}(\Lambda);
\end{align}
see Lemma \ref{lemma_circ} below. We call $\newD$ an upper density because, due to the sublinearity of $\limsup$,
\begin{align}\label{eq_sub}
\newD(\Lambda_1 \cup \Lambda_2) \leq \newD(\Lambda_1) + \newD(\Lambda_2),
\end{align}
for any two sets $\Lambda_1, \Lambda_2 \subseteq \mathbb{R}$.

Below we prove the following uniqueness result formulated in terms of the upper average circular density of the zero set $\{f=0\}$ of a function $f$ (counted without multiplicities).

\begin{theorem}[Uniqueness theorem]\label{th_unique}
Let $g$ be a totally positive function of Gaussian type, as in \eqref{eq_tpgt}.
Let $f \in V^\infty(g)$ be non-zero. Then
$\newD(\{f=0\}) \leq 1$.
\end{theorem}
The uniqueness theorem (which applies also to non-real-valued functions) allows one to carry out the following natural proof of the sign retrieval theorem.

\begin{proof}[Proof of Theorem \ref{th_main}, assuming Theorem \ref{th_unique}]
	Assume that $D^{-}(\Lambda)>2$ and write $\Lambda = \Lambda_1 \cup \Lambda_2$ as in \eqref{eq_split}. Then, by \eqref{eq_dens_3} and \eqref{eq_sub}, either $\newD(\Lambda_1) >1$ or $\newD(\Lambda_2) >1$ (possibly both). In the first case, Theorem \ref{th_unique} shows that $f_1 \equiv f_2$, while in the second, $f_1 \equiv -f_2$.
\end{proof}

As the proof shows, Theorem \ref{th_main} remains valid if \eqref{eq_d_2} is relaxed to $\newD(\Lambda)>2$. Sign retrieval also holds under the same density condition for the shift-invariant spaces $V^p(g)$ defined with respect to $L^p$ norms, $1 \leq p \leq \infty$, as these are contained in $V^\infty(g)$.

Towards the proof of Theorem \ref{th_unique}, we first prove \eqref{eq_dens_2} and \eqref{eq_dens_3}.

\begin{lemma}\label{lemma_circ}
Let $\Lambda \subseteq \mathbb{R}$ and $\alpha>0$. Then
\eqref{eq_dens_2} and \eqref{eq_dens_3} hold true.
\end{lemma}
\begin{proof}
We assume that $\Lambda$ has no accumulation points, since, otherwise, both sides of \eqref{eq_dens_2} are infinite. Denote provisionally the right hand side of \eqref{eq_dens_2} by $\newDD(\Lambda)$, and set $\Lambda' := \Lambda \setminus \{0\}$.

\noindent \emph{Step 1}. Let $\varepsilon \in (0,1)$,
and $C=C_{\alpha,\varepsilon}>0$ a constant to be specified. We claim:
\begin{align}
\label{eq_a1}
\newD(\Lambda)& = 
\limsup_{r \longrightarrow \infty} \frac{4}{\pi r^2}
\int_{C} ^r \sum_{\lambda \in \Lambda' \cap [-t,t]} 
\sqrt{t^2 - \lambda^2} \,
\frac{dt}{t},
\\
\label{eq_a2}
(1+\delta)^{2}
\newDD(\Lambda) &= \limsup_{r \longrightarrow \infty}
\frac{2\alpha}{\pi r^2} \int_{C}^r \# \big[ (\Lambda' \times \alpha \mathbb{Z}) \cap B_{(1+\delta)t}(0) \big] \,\frac{dt}{t}, \qquad \delta=\pm \varepsilon.
\end{align}
To prove the first claim, we first note that $\newD(\Lambda)=\newD(\Lambda')$, since, if $0 \in \Lambda$, the contribution of the point $\lambda=0$ is
\begin{align*}
\frac{4}{\pi r^2}
\int_0^r dt = \frac{4}{\pi r},
\end{align*}
and does not affect the limit on $r$. Similarly, since $\sqrt{t^2 - \lambda^2} \leq t$,
\begin{align*}
\frac{4}{\pi r^2}
\int_0^C \sum_{\lambda \in \Lambda' \cap [-t,t]} 
\sqrt{t^2 - \lambda^2}
\frac{dt}{t} 
\leq  \# \big( \Lambda' \cap [-C,C]\big) \cdot \frac{4C}{\pi r^2}.
\end{align*}
The last quantity is finite because $\Lambda$ has no accumulation points. This proves \eqref{eq_a1}, because it shows that limiting the integral to $[C,r]$ does not affect the limit on $r$. Second, a change of variables shows that \eqref{eq_a2} holds with $C$ replaced by $0$. In addition, since $\Lambda$ has no accumulation points, there exists $\eta>0$ such that $\Lambda'\cap(-\eta,\eta)=\emptyset$. Therefore, we can estimate
the part of the integral excluded in \eqref{eq_a2} as
\begin{align*}
\frac{2\alpha}{\pi r^2} \int_{0}^C \# \big[(\Lambda' \times \alpha \mathbb{Z}) \cap B_{(1+\delta)t}(0) \big]\,\frac{dt}{t}
\leq \# \big[ (\Lambda' \times \alpha \mathbb{Z}) \cap B_{2C}(0) \big]
\cdot \frac{2\alpha}{\pi r^2} \int_{\frac{\eta}{1+\varepsilon}}^C \frac{dt}{t},
\end{align*}
which is finite because $\eta>0$. This proves \eqref{eq_a2}.

\noindent \emph{Step 2}. For $t>0$, we note that
\begin{align*}
\# (\Lambda' \times \alpha \mathbb{Z}) \cap B_t(0)
&= \sum_{\lambda \in \Lambda' \cap [-t,t]} 
\# \left\{k \in \mathbb{Z}: \lambda^2+\alpha^2 k^2 < t^2 \right\}
\\
&= \sum_{\lambda \in \Lambda' \cap [-t,t]} 
\# \left\{k \in \mathbb{Z}: \abs{k}  < \alpha^{-1} \sqrt{t^2 - \lambda^2} \right\}.
\end{align*}
Next, we choose the constant $C_{\alpha, \varepsilon}$ so that the following estimates hold for $t \geq C_{\alpha, \varepsilon}$:
\begin{align*}
&\# \big[ (\Lambda' \times \alpha \mathbb{Z}) \cap B_{(1+\varepsilon)t}(0) \big]
\geq 
\sum_{\lambda \in \Lambda' \cap [-t,t]} 
\left[
\frac{2}{\alpha} \sqrt{(1+\varepsilon)^2 t^2 - \lambda^2} - 1\right]
\\
&\qquad=\sum_{\lambda \in \Lambda' \cap [-t,t]} 
\frac{2}{\alpha} \sqrt{(1+\varepsilon)^2 t^2 - \lambda^2
- \alpha \sqrt{(1+\varepsilon)^2 t^2 - \lambda^2} + \alpha^2/4
}
\\
&\qquad\geq\sum_{\lambda \in \Lambda' \cap [-t,t]} 
\frac{2}{\alpha} \sqrt{(1+\varepsilon)^2 t^2 - \lambda^2
	- \alpha (1+\varepsilon)t + \alpha^2/4
}
\\
&\qquad\geq 
\sum_{\lambda \in \Lambda' \cap [-t,t]} 
\frac{2}{\alpha} \sqrt{t^2 - \lambda^2},
\end{align*}
where the last estimate requires choosing $C_{\alpha, \varepsilon}$ large. Similarly,
\begin{align*}
&\# \big[(\Lambda' \times \alpha \mathbb{Z}) \cap B_{(1-\varepsilon)t}(0)\big]
\leq
\sum_{\lambda \in \Lambda' \cap [-(1-\varepsilon)t,(1-\varepsilon)t]} 
\left[
\frac{2}{\alpha} \sqrt{(1-\varepsilon)^2 t^2 - \lambda^2} + 1\right]
\\
&\qquad=
\sum_{\lambda \in \Lambda' \cap [-(1-\varepsilon)t,(1-\varepsilon)t]} 
\frac{2}{\alpha} \sqrt{(1-\varepsilon)^2 t^2 - \lambda^2
+ \alpha\sqrt{(1-\varepsilon)^2 t^2 - \lambda^2}
+\alpha^2/4
}
\\
&\qquad\leq
\sum_{\lambda \in \Lambda' \cap [-(1-\varepsilon)t,(1-\varepsilon)t]} 
\frac{2}{\alpha} \sqrt{(1-\varepsilon)^2 t^2 - \lambda^2
	+ \alpha {(1-\varepsilon) t}
	+\alpha^2/4
}
\\
&\qquad \leq
\sum_{\lambda \in \Lambda' \cap [-t,t]} 
\frac{2}{\alpha} \sqrt{t^2 - \lambda^2},
\end{align*}
where the last estimate requires choosing $C_{\alpha, \varepsilon}$ large.
Combining the last two estimates with \eqref{eq_a1} and \eqref{eq_a2}, and letting $\varepsilon \longrightarrow 0$,
we deduce \eqref{eq_dens_2}.

\noindent \emph{Step 3}. To prove \eqref{eq_dens_3}, let $r>s>0$ and use \eqref{eq_dens_2} with $\alpha=1$ to estimate
\begin{align*}
\newD(\Lambda) &\geq \liminf_{r \longrightarrow \infty}
\frac{2}{\pi r^2}
\int_s^r \inf_{z \in \mathbb{R}^2}
\# \big[ (\Lambda' \times \mathbb{Z}) \cap B_t(z) \big] \,\frac{dt}{t}
\\
& \geq \inf_{t \geq s} 
\left[
\tfrac{1}{\abs{B_t(0)}} \inf_{z \in \mathbb{R}^2}
\# \big[ (\Lambda' \times \mathbb{Z}) \cap B_t(z)\big] \right]
\liminf_{r \longrightarrow \infty}
\frac{2}{\pi r^2}
\int_s^r \abs{B_t(0)} \,\frac{dt}{t}
\\
&=\inf_{t \geq s} 
\left[
\inf_{z \in \mathbb{R}^2} \tfrac{1}{\abs{B_t(z)}} 
\# \big[ (\Lambda' \times \mathbb{Z}) \cap B_t(z) \big] \right].
\end{align*}
Letting $s \longrightarrow \infty$, we see that
\begin{align*}
\newD(\Lambda) \geq \liminf_{t \longrightarrow \infty}
\inf_{z \in \mathbb{R}^2} \tfrac{1}{\abs{B_t(z)}} 
\# \big[ (\Lambda' \times \mathbb{Z}) \cap B_t(z) \big] = D^{-}(\Lambda') = D^{-}(\Lambda),
\end{align*}
where we used that the two-dimensional lower Beurling density of 
$\Lambda' \times \mathbb{Z}$ is $D^{-}(\Lambda)$.
\end{proof}

We can now prove the uniqueness Theorem. The proof builds on \cite[Lemma 5.1]{GRS18} and \cite[Theorems 4.3 and 4.9]{GRS20}.

\begin{proof}[Proof of Theorem \ref{th_unique}]
	We proceed by induction on $m$ in \eqref{eq_tpgt}.
	
	\smallskip
	
	\noindent \emph{In the case $m=0$}, the function $g$ is a Gaussian $g(x) = C'_0 e^{-a x^2}$, with $a=\frac{\pi^2}{\gamma}>0$. Let
	$f=\sum_k c_k g(\cdot-k)\in V^\infty (g)$ with $c$ bounded
	be non-zero, and denote by $\Lambda \subseteq \mathbb{R}$ its zero set.
	
	As shown in \cite[Lemma 4.1]{GRS20}, $f$ possesses an extension to an
	entire function on $\mathbb{C}$, satisfying the growth estimate
	$|f(x+iy)| \lesssim e^{a y^2}$ for $x,y \in \bR$,
	where the implied constant depends on $f$.
	Let $n \geq 0$ be the order of $f$ at $z=0$, and consider the analytic function \[F(z) := C_1 z^{-n} f(z) e^{\frac{a}{2} z^2},\]
	where $C_1 \in \mathbb{C}$ is chosen so that $F(0)=1$. Then $F$ satisfies
	\begin{align}\label{eq_growth}
	\abs{F(x+iy)} \leq C e^{ay^2} e^{\frac{a}{2} (x^2-y^2)} = C e^{\frac{a}{2} (x^2+y^2)}, \qquad x,y \in \mathbb{R},
	\end{align}
	for some constant $C>0$.
	Moreover, the zero set of $f$ is invariant under addition of $i\tfrac{\pi}{a}\mathbb{Z}$ \cite[Lemma 4.2]{GRS20}. Therefore, 
	the set of complex zeros of $F$ contains $\big(\Lambda \setminus\{0\}\big) + i\tfrac{\pi}{a}\mathbb{Z}$:
	\begin{align*}
	F \big(\lambda + i \tfrac{\pi}{a} k \big) = 0,
	\qquad \lambda \in \Lambda \setminus \{0\}, k \in \mathbb{Z}.
	\end{align*}
	By \eqref{eq_dens_2} with $\alpha=\tfrac{\pi}{a}$, the zero-counting function of $F$,
	\begin{align*}
	n_F(t) :=
	\# \{ z \in \mathbb{C}: F(z)=0, \abs{z} \leq t \},
	\end{align*}
	satisfies
	\begin{align}\label{eq_1}
	\limsup_{r \longrightarrow \infty}
	\frac{1}{r^2} 
	\int_0^r \frac{n_F(t)}{t} dt \geq \frac{a}{2} \newD(\Lambda).
	\end{align}
	On the other hand,
	by Jensen's formula combined with \eqref{eq_growth}, for all $r>0$,
	\begin{align}\label{eq_2}
	\frac{1}{r^2} 
	\int_0^r \frac{n_F(t)}{t} dt =
	\frac{1}{2\pi r^2} \int_0^{2\pi} \log|{F\big(r
		e^{i\theta}\big)}| d\theta \leq \frac{\log(C)}{r^2} + \frac{a}{2}.
	\end{align}
	By combining \eqref{eq_1} and \eqref{eq_2} we conclude that $\newD(\Lambda) \leq 1$, as claimed. Note, in particular, that $\Lambda$ has no accumulation points.
	
	\bigskip
	
	\noindent \emph{Inductive step.} Let $g$ be as in \eqref{eq_tpgt} and $f=\sum_k c_k g(\cdot-k) \in V^\infty(g)$ non-zero. Without loss of generality we assume that, in \eqref{eq_tpgt}, $\delta_1, \ldots, \delta_m$ are non-zero.	
	Since $g$ is real-valued, by replacing $c_k$ with
	$\Re(c_k)$ or $\Im(c_k)$ if necessary, we may assume that $f$ is real-valued. Let $g_1$ be given by \eqref{eq_tpgt}, except that the product runs only up to $m-1$. Then
	\begin{align*}
	f_1 := f + \delta_m f' = \sum_{k\in\mathbb{Z}} c_k g_1(\cdot-k)
	\end{align*}
	belongs to $V^\infty(g_1)$ and is real-valued and non-zero, because the coefficients $c_k$ are real and at least one of them is non-zero.
	We assume by inductive hypothesis that $\newD(\{f_1=0\}) \leq 1$.
	
		Let $\Lambda=\{f=0\}$ and $\Gamma=\{f_1=0\}$. Suppose that $\Lambda$ has no accumulation points. Assume for the moment that, in addition, $f$ has a non-negative zero, and order the set of non-negative zeros of $f$ increasingly:
	\begin{align*}
	\Lambda \cap [0,\infty)=\{\lambda_k : k=0,\ldots,N\},
	\end{align*}
	where $N \in \mathbb{N}\cup\{0,\infty\}$, and we do not count multiplicities. Hence, $\lambda_k >0$, if $k>0$.
	
	By Rolle's Theorem applied to the differential operator $I + \delta_m \partial_x$, there is a sequence of zeros of $f_1$ that interlaces $\Lambda \cap [0,\infty)$. (Indeed, we note that $\frac{d}{dx} \big[ e^{x/\delta_m} f(x)\big] = \frac{1}{\delta_m} e^{x/\delta_m} f_1(x)$ and apply the standard version of Rolle's theorem; see \cite[Lemma 5.1]{GRS18} or \cite[Lemma 4.8]{GRS20} for details.) We parameterize  that sequence as
	$\Gamma^+=\big\{\gamma_k : k=1,\ldots,N\}$, with
	\begin{align*}
	\lambda_{k-1} < \gamma_k < \lambda_k, \qquad k>0.
	\end{align*}
	The set $\Gamma^+$ is empty if $N=0$.	For $t>0$, this ordering associates with each $\lambda_k \in (0,t]$, $k \not=0$, a distinct zero of $f_1$, $\gamma_k \in (0,t]$, such that
	the vertical segment through $\gamma_k$ contained in the disk $B_t(0)$ is longer than the corresponding segment through $\lambda_k$:
	\begin{align*}	
	\sum_{k >0, \lambda_k \in (0,t] } \sqrt{t^2 - \lambda_k^2}
	&\leq \sum_{k >0, \gamma_k \in (0,t] } \sqrt{t^2 - \gamma_k^2}
	\leq \sum_{\gamma \in \Gamma \cap (0,t] } \sqrt{t^2 - \gamma^2};
	\end{align*}
	see Figure \ref{fig}.	
	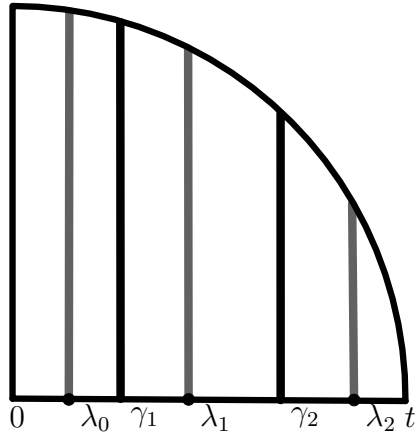
\begin{figure}[ht]
		\definecolor{wrwrwr}{rgb}{0.3803921568627451,0.3803921568627451,0.3803921568627451}
		\scalebox{.4}{
			\begin{tikzpicture}[line cap=round,line join=round,>=triangle 45,x=1cm,y=1cm]
			\clip(-10.015451103368168,-6.401813472706154) rectangle (11.716953077816502,9.608639488966313);
			\draw [line width=8pt,color=wrwrwr] (-6.300762774187206,8)-- (-6.3194,-4.9078);
			\draw [line width=8pt] (-4.608688292274559,7.64)-- (-4.5994,-4.8878);
			\draw [line width=8pt,color=wrwrwr] (-2.353670546206893,6.75)-- (-2.3394,-4.9078);
			\draw [line width=8pt] (0.7179658342394823,4.60)-- (0.7206,-4.9278);
			\draw [line width=8pt,color=wrwrwr] (3.097243468917373,1.6)-- (3.1606,-4.9278);
			\draw [fill=black] (-8.19,-4.88) circle (2.5pt);
			\draw[color=black] (-8.029388385598132,-5.5) node{\scalebox{2.5} {0}};
			\draw [fill=black] (4.87,-4.96) circle (2.5pt);
			\draw[color=black] (5.032632520325212,-5.5) node {\scalebox{2.5}{$t$}};
			\draw [fill=black] (-6.3194,-4.9078) circle (5.5pt);
			\draw[color=black] (-5.433337282229956,-5.5) node {\scalebox{2.5}{{$\lambda_0$}}};
			\draw [fill=black] (-4.5994,-4.8878) ++(-2.5pt,0 pt) -- ++(2.5pt,2.5pt)--++(2.5pt,-2.5pt)--++(-2.5pt,-2.5pt)--++(-2.5pt,2.5pt);
			\draw[color=black] (-3.798029500580711,-5.5) node {\scalebox{2.5}{$\gamma_1$}};
			\draw [fill=black] (-2.3394,-4.9078) circle (5.5pt);
			\draw[color=black] (-1.4472745644599214,-5.5) node {\scalebox{2.5}{$\lambda_1$}};
			\draw [fill=black] (0.7206,-4.9278) ++(-2.5pt,0 pt) -- ++(2.5pt,2.5pt)--++(2.5pt,-2.5pt)--++(-2.5pt,-2.5pt)--++(-2.5pt,2.5pt);
			\draw[color=black] (1.5371621370499509,-5.5) node {\scalebox{2.5}{$\gamma_2$}};
			\draw [fill=black] (3.1606,-4.9278) circle (5.5pt);
			\draw[color=black] (4.0514478513356655,-5.5) node {\scalebox{2.5}{$\lambda_2$}};
			\draw [shift={(-8.19,-4.88)},line width=6pt]  (0,0) --  plot[domain=-0.006125497658366008:1.5707963267948966,variable=\t]({1*13.060245020672467*cos(\t r)+0*13.060245020672467*sin(\t r)},{0*13.060245020672467*cos(\t r)+1*13.060245020672467*sin(\t r)}) -- cycle ;
			\end{tikzpicture}}
		\caption{Each positive zero $\lambda_k$ of $f$ is assigned to a zero $\gamma_k$ of $f_1$ in such a way that
			the corresponding vertical segments contained in the disk $B_t(0)$ become longer.}
		\label{fig}
	\end{figure}
	We now add the term corresponding to $\lambda_0$, and bound $\sqrt{t^2 - \lambda_0^2} \leq t$ to obtain 
	\begin{align*}
	\sum_{\lambda \in \Lambda \cap [0,t] } \sqrt{t^2 - \lambda^2}
	\leq t + 
	\sum_{\gamma \in \Gamma \cap (0,t] } \sqrt{t^2 - \gamma^2}.
	\end{align*}
	The previous estimate is also trivially true is $f$ has no non-negative zeros.
	Arguing similarly with the non-positive zeros of $f$ we conclude that
	\begin{align*}
	\sum_{\lambda \in \Lambda \cap [-t,t] } \sqrt{t^2 - \lambda^2}
	\leq 2t + 
	\sum_{\gamma \in \Gamma \cap [-t,t] } \sqrt{t^2 - \gamma^2}.
	\end{align*}
	This shows that $\newD(\{f=0\}) \leq \newD(\{f_1=0\})$. The latter quantity is bounded by $1$ by inductive hypothesis. Finally, if $\Lambda=\{f=0\}$ has an accumulation point, Rolle's theorem applied as before shows that so does $\Gamma=\{f_1=0\}$, and, therefore,
	$\newD(\{f_1=0\}) = \infty$ contradicting the inductive hypothesis.
\end{proof}

\definecolor{wwwwww}{rgb}{0.4,0.4,0.4}
\definecolor{ududff}{rgb}{0.30196078431372547,0.30196078431372547,1}
\definecolor{xdxdff}{rgb}{0.49019607843137253,0.49019607843137253,1}

\end{document}